\documentclass[a4paper]{elsarticle}
\usepackage{amsfonts,anysize,amsmath,indentfirst,geometry}
\usepackage{bm}
\marginsize{30mm}{25mm}{18mm}{18mm}
\usepackage{lineno}
\usepackage[colorlinks,citecolor=blue,linkcolor=blue]{hyperref}
\usepackage{mathrsfs}
\usepackage{amssymb}
\usepackage{array}
\usepackage{graphicx}
\usepackage{booktabs}
\usepackage{threeparttable}
\usepackage{multirow}
\usepackage{epstopdf}
\usepackage{algorithm}
\usepackage{algorithmic}
\numberwithin{equation}{section}
\newtheorem{theorem}{Theorem}[section]
\newtheorem{lemma}[theorem]{Lemma}
\newdefinition{remark}{Remark}[section]
\newdefinition{corollary}{Corollary}[section]
\newdefinition{example}{Example}[section]
\newproof{proof}{\textbf{Proof}}
\newproof{pot}{Proof of Theorem \ref{thm2}}
\linenumberdisplaymath
\modulolinenumbers[1]
\journal{Manuscript}
\begin{document}

\begin{frontmatter}
\title{A note on the condition number of the scaled total least squares problem\tnoteref{t1}}
\tnotetext[t1]{This work is supported by a project of Shandong Province Higher Educational Science and Technology Program (Grant No. J17KA160) and the National Natural Science Foundation of China (Grant Nos. 11671059,11671060)}
\author[a]{Shaoxin Wang\corref{cor1}}
\ead{ shwangmy@163.com, shxwang@qfnu.edu.cn}
\author[b]{Hanyu Li}
\ead{lihy.hy@gmail.com}
\author[b]{Hu Yang}
\ead{hy@cqu.edu.cn}
\cortext[cor1]{Corresponding author}
\address[a]{School of Statistics, Qufu Normal University, \\ Qufu, 273165, P. R. China}
\address[b]{College of Mathematics and Statistics, Chongqing University, \\ Chongqing, 401331, P. R. China}

\begin{abstract}
  In this paper, we consider the explicit expressions of the normwise condition number for the scaled total least squares problem. Some techniques are introduced to simplify the expression of the condition number, and some new results are derived. Based on these new results, new expressions of the condition number for the total least squares problem can be deduced as a special case. New forms of the condition number enjoy some storage and computational advantages. We also proposed three different methods to estimate the condition number. Some numerical experiments are carried out to illustrate the effectiveness of our results.
\end{abstract}
\end{frontmatter}


\section{Introduction}
The scaled total least squares (STLS) problem (or technique) was proposed by Rao \cite{Rao97} to give a unified treatment of ordinary least squares (OLS) problem, total least squares (TLS) problem and the data least squares (DLS) problem. Paige and Strak\v{o}s \cite{Paig02} reformulated the STLS problem and presented a detailed analysis of conditions that guarantee the STLS problem has a unique solution. Following their line, for $A\in\mathbb{R}^{m\times n}$ with $m>n$ and $b\in\mathbb{R}^m$, the STLS problem is given by
\begin{equation}\label{stls}
    \min \left\|\begin{bmatrix}
                  E & r \\
                \end{bmatrix}\right\|_F,\textrm{ subject to }\lambda b-r \in \mathcal{R}(A+E),
\end{equation}
where $\lambda$ is a positive real number, $\|\cdot\|_F$ denotes the Frobenius norm and $\mathcal{R}(\cdot)$ is the range space. Let
$[E_{S}\; r_{S}]$ be the solution to \eqref{stls}, then the solution to the linear system $(A+E_{S})\lambda x=\lambda b-r_{S}$ is called the STLS solution and denoted by $x_{S}$. As shown in \cite{Paig02}, when $\lambda=1$, $\lambda\rightarrow 0$ and $\lambda\rightarrow \infty$, $x_{S}$ becomes the TLS solution $x_{\intercal}$, OLS solution $x_\mathrm{O}$ and DLS solution $x_D$, respectively.

The condition number gives a quantitative measurement of the maximum amplification of the resulting change in solution with respect to a perturbation in the data and has been extensively studied for too many topics to list here. For the STLS problem, Zhou et al.~\cite{Zhou09} considered its perturbation analysis and presented the normwise, mixed and componentwise condition numbers. Based on the perturbation theory of singular value decomposition (SVD) given in \cite{SunJ88}, Li and Jia \cite{LiJ11} gave a different approach to derive the normwise and componentwise condition numbers of the STLS problem, and the structured condition numbers were also investigated there. It should be noted that the normwise condition number in \cite{Zhou09} contains a Kronecker product which makes it impractical to compute, especially for large-scale problems. Based on the fact that $\|A\|_2=\|A^TA\|_2^{1/2}$, $\|\cdot\|_2$ denotes the spectral norm of matrix or Euclidean norm of vector, some closed formulas and upper (or lower) bounds of normwise condition number for the TLS problem were given in \cite{BabG11} and \cite{JiaL13}, and these results are easy to compute and do not contain Kronecker product any more. Xie et al.~\cite{Xie14} showed that the expressions of condition number given in \cite{BabG11} and \cite{JiaL13} are mathematically equivalent. However, computing the matrix cross product $A^TA$ is a source of rounding errors and is potentially numerical unstable \cite[pp.~386]{High02}. Some progress to avoid computing $A^TA$ was made in \cite{BabG11, JiaL13, Xie14}. In this paper, we present a new expression of the normwise condition number for the STLS problem. The new expression is easy to compute and does not need to compute Kronecker product or matrix cross product. On the other hand, we also propose a procedure to compute the condition number, which does not need to form the explicit expression of condition number and avoid computing the Kronecker product.

The rest of the paper is organized as follows. Section \ref{Sect2} contains the main results of the paper. Some algorithms and numerical experiments are presented in Section \ref{Sect3}. Concluding remarks are given in Section \ref{Sect4}. Before proceeding to the following sections, we introduce some notation first: For any matrix $B$, $A\otimes B=[a_{ij}B]$ denotes the Kronecker product of $A$ and $B$, $\mathrm{vec}(\cdot)$ is a linear map defined by $\mathrm{vec}(A)=[a_{1,1},\cdots, a_{m,1},\cdots, a_{1,n},\cdots, a_{m,n}]^T$.

\section{Main results}
\label{Sect2}
As stated in the Introduction, the TLS can be treated as a special case of the STLS problem. An interesting result is that we can solve the STLS problem by finding the solution to a special TLS problem. When $\lambda=1$, we get the following TLS problem
\begin{equation}\label{tls}
    \min \left\|\begin{bmatrix}
                  E & r \\
                \end{bmatrix}\right\|_F,\textrm{ subject to } b-r \in \mathcal{R}(A+E).
\end{equation}
It is easy to check that when $x_{S}$ is the solution of \eqref{stls} then $\lambda x_{S}$ is the TLS solution to the following TLS problem
\begin{eqnarray}\label{sptls}
    \min \left\|\begin{bmatrix}
                  E & r \\
                \end{bmatrix}\right\|_F,\textrm{ subject to } \lambda b-r \in \mathcal{R}(A+E).
\end{eqnarray}
Let the SVDs of the matrices $[A \; \lambda b]$ and $A$ be
\begin{eqnarray*}
    U^T\begin{bmatrix}
         A & \lambda b \\
       \end{bmatrix}V=\Sigma,\quad
        \hat{U}^TA\hat{V}=\hat{\Sigma},
\end{eqnarray*}
where $U=[u_1, \cdots, u_{n+1}]\in \mathbb{R}^{m\times (n+1)}$, $V=[v_1,\cdots, v_{n}]\in \mathbb{R}^{(n+1)\times (n+1)}$, $\Sigma=\mathrm{diag}(\sigma_1,\cdots,\sigma_{n+1})$ with
           $\sigma_1\geq, \cdots,\geq \sigma_{n+1}\geq 0$, and
$\hat{\Sigma}=\mathrm{diag}(\hat{\sigma}_1,\cdots,\hat{\sigma}_{n})$
           with $\hat{\sigma}_1\geq, \cdots,\hat{\sigma}_{n}\geq 0$.
Analogous to the Golub-Van Loan condition \cite{GolV80} for TLS problem to guarantee the existence and uniqueness of a solution, Zhou et al. \cite{Zhou09} presented the following sufficient condition to ensure the STLS problem has a unique solution
\begin{equation}
\label{sfcd}
    \hat{\sigma}_{n}>\sigma_{n+1}.
\end{equation}
Therefore, by \eqref{sfcd} the TLS solution to \eqref{sptls} is
\begin{eqnarray*}
    \lambda x_{S}=(A^TA-\sigma_{n+1}I_n)^{-1}A^T(\lambda b),
\end{eqnarray*}
which gives
\begin{eqnarray}
\label{xS}
     x_{S}=(A^TA-\sigma_{n+1}I_n)^{-1}A^Tb.
\end{eqnarray}

Let $\Delta A$ and $\Delta b$ be the corresponding perturbations to $A$ and $b$, then we have the following perturbed STLS problem
\begin{equation}\label{pstls}
    \min \left\|\begin{bmatrix}
                  E & r \\
                \end{bmatrix}\right\|_F,\textrm{ subject to }\lambda (b+\Delta b)-r \in \mathcal{R}\left((A+\Delta A)+E\right).
\end{equation}
For the perturbed STLS problem, Zhou et al. \cite{Zhou09} and Li and Jia \cite{LiJ11} presented two different approaches to show that when the perturbation $[\Delta A \; \Delta b]$ is sufficiently small, the perturbed STLS problem admits a unique solution $x_{PS}$. We take the result given in \cite{LiJ11} as  the following theorem with some modifications of symbols.
\begin{theorem}
\label{Thm1}
Under the assumption \eqref{sfcd}, if $\|[\Delta A\; \Delta b]\|_F$ is small enough, then the perturbed STLS problem \eqref{pstls} has the unique solution $x_{PS}$. Moreover,
\begin{eqnarray}\label{dltX}
\Delta x=x_{PS}-x_S= K\begin{bmatrix}
            \mathrm{vec}(\Delta A) \\
            \Delta b \\
          \end{bmatrix}+\mathcal{O}\left(\left\|\left[\Delta A, \Delta b\right]\right\|_F^2\right)
\end{eqnarray}
where
\begin{eqnarray*}
    K=M^{-1}\left(\left(\frac{2}{\|r\|_2^2}A^Trr^T-A^T\right)
\left(\begin{bmatrix}
  x_S^T & -1 \\
\end{bmatrix}\otimes I_m\right)-\begin{bmatrix}
                                  I_n\otimes r^T & 0_{n\times m} \\
                                \end{bmatrix}\right)
\end{eqnarray*}
with $M=A^TA-\sigma^2_{n+1}I_n$ and $r=Ax_S-b$,
\end{theorem}
Li and Jia \cite{LiJ11} also presented a vary detailed comparison of the above results and those given in \cite{Zhou09} and showed that their perturbation estimate is the same as that given in \cite{Zhou09}. In addition, according to Theorem \ref{Thm1} we can deduce that if we set
\begin{eqnarray*}
  F: \mathbb{R}^{m\times n}\times \mathbb{R}^{m} & \rightarrow & \mathbb{R}^{n} \\
  \begin{bmatrix}
    A & \lambda b \\
  \end{bmatrix}
  &\rightarrow & x_{S}=M^{-1}A^Tb
\end{eqnarray*}
then the map $F$ is Fr\'{e}chet differentiable at $[A\; \lambda b]$ under the assumption \eqref{sfcd} and the Fr\'{e}chet derivative of $F$ at $[A\; \lambda b]$ is given by
\begin{eqnarray*}
    DF(A,\lambda b):=K.
\end{eqnarray*}

According to the definition of condition number given in \cite{Rice66} and \cite{Geurt82}, the relative normwise condition number of STLS problem is given by
\begin{equation}\label{defcd}
    \kappa_{rF}(A,\lambda b)=\lim_{\delta\rightarrow 0}\sup_{\left\|[\Delta A \; \lambda \Delta b]\right\|_F<\delta}\frac{\frac{\left\|F\left(A+\Delta A,\lambda (b+\Delta b)\right)-F(A, \lambda b)\right\|_2}{\|F(A,\lambda b)\|_2}}{\frac{\left\|[\Delta A \; \lambda \Delta b]\right\|_F}{\left\|[A \; \lambda b]\right\|_F}}.
\end{equation}
When $F$ is Fr\'{e}chet differentiable, $\kappa_F(A,\lambda b)$ reduces to
\begin{eqnarray*}
    \kappa_{rF}(A,\lambda b)=\frac{\left\|DF\left(A,\lambda b\right)\right\|_2\left\|[A \; \lambda b]\right\|_F}{\|F(A,\lambda b)\|_2},
\end{eqnarray*}
and $\kappa_{F}(A,\lambda b)=\left\|DF\left(A,\lambda b\right)\right\|_2$ is the absolute condition number. We should remark that the definition of condition number given by \eqref{defcd} can be extended to a more general sense. Wang and Yang \cite{Wang16} presented a unified definition of condition number to cope with the conditioning of equality constrained indefinite least squares problem, which include the normwise, mixed and componentwise condition numbers as its special cases, for further discussions see \cite{Wang16,LiW16}.

For the convenience of presentation, we summarize the above discuss in the following theorem.
\begin{theorem}\label{Thm2}
Under the assumption \eqref{sfcd}, the relative normwise condition number of STLS problem defined by \eqref{defcd} is
\begin{eqnarray*}
    \kappa_{rF}(A,\lambda b)=\frac{\left\|K\right\|_2\left\|[A \; \lambda b]\right\|_F}{\|F(A,\lambda b)\|_2},
\end{eqnarray*}
and its absolute condition number is given by
\begin{equation}
\label{eqblk0}
    \kappa_{F}(A,\lambda b)=\left\|K\right\|_2.
\end{equation}
\end{theorem}

It should be noted that the Kronecker product enlarges the order of matrix and may make it impractical to explicitly forming $K$ when $m$ and $n$ are large. For the TLS problem, adjoint techniques are employed to eliminate the Kronecker product in \cite{Xie14, BabG11}. Following their step and going to further, we give the following theorem to simplify the condition number of the STLS problem, which is also the main result of our paper.

Considering the relationship between relative and absolute condition numbers, in the following we only focus on the simplification of $\kappa_{F}(A,\lambda b)$.
\begin{theorem}
\label{Thm.main}
The absolute condition number $\kappa_{F}(A,\lambda b)$ for STLS problem has the following two equivalent forms
\begin{equation}\label{eqform1}
    \kappa_{F1}(A,\lambda b)=\left\|M^{-1}\left((1+\|x_S\|^2_2)A^TA-A^Trx_S^T-x_Sr^TA+\|r\|_2^2I_n\right)M^{-1}\right\|_2^{\frac{1}{2}},
\end{equation}
and
\begin{eqnarray}\label{eqform2}
    \kappa_{F2}(A,\lambda b)=\left\|M^{-1}\begin{bmatrix}
        A^T, & \|x_S\|_2A^T\left(I_m-\frac{1}{\|r\|_2^2}rr^T\right), & \|r\|_2\left(I_n-\frac{1}{\|r\|_2^2}A^Trx_S^T\right) \\
      \end{bmatrix}\right\|_2.
\end{eqnarray}
\end{theorem}
\begin{proof}
For a real matrix $X$, $\|X\|_2=\|X^TX\|_2^{1/2}=\|XX^T\|_2^{1/2}$ holds. Analogous to \cite{BabG11}, we have
\begin{eqnarray*}
    \left\|K\right\|_2=\left\|KK^T\right\|_2^{\frac{1}{2}}.
\end{eqnarray*}
Since $M$ is symmetric, by the equality $\mathrm{vec}(AXB)=(B^T\otimes A)\mathrm{vec}(X)$
\cite[Chapt.4]{Horn91} we can get
\begin{eqnarray}
    KK^T&=&M^{-1}\left(\left(\frac{2}{\|r\|_2^2}A^Trr^T-A^T\right)\left(\begin{bmatrix}
  x_S^T & -1 \\
  \end{bmatrix}\otimes I_m\right)-\begin{bmatrix}
                                  I_n\otimes r^T & 0_{n\times m} \\                               \end{bmatrix}\right)\nonumber\\
  & &\quad\times \left(\left(\begin{bmatrix}
    x_S \\
   -1 \\
  \end{bmatrix}\otimes I_m\right)\left(\frac{2}{\|r\|_2^2}rr^TA-A\right)-\begin{bmatrix}
                                  I_n\otimes r \\
                                  0_{m\times n} \\                               \end{bmatrix}\right)M^{-1}\nonumber\\
  &=&M^{-1}\left((1+\|x_S\|^2_2)A^TA-A^Trx_S^T-x_Sr^TA+\|r\|_2^2I_n\right)M^{-1}\label{eqblk5}\\
  &=&M^{-1}\left(\begin{bmatrix}
                 A^T & I_n \\
               \end{bmatrix}
               \begin{bmatrix}
                 \left(1+\|x_S\|_2^2\right)I_m & -rx_S^T \\
                 -x_Sr^T & \|r\|_2^2I_n \\
               \end{bmatrix}
               \begin{bmatrix}
                 A \\
                 I_n \\
               \end{bmatrix}
  \right)M^{-1}.\label{eqblk}
\end{eqnarray}
Since
\begin{eqnarray}
    \begin{bmatrix}
\left(1+\|x_S\|_2^2\right)I_m & -rx_S^T \\
-x_Sr^T & \|r\|_2^2I_n \\
\end{bmatrix}&=&\begin{bmatrix}
                I_m & -\frac{1}{\|r\|_2^2}rx_S^T \\
                0_{n\times m} & I_n \\
              \end{bmatrix}
              \begin{bmatrix}
                I_m+\|x_S\|_2^2\left(I_m-\frac{1}{\|r\|_2^2}rr^T\right) & 0_{m\times n} \\
                0_{n\times m} & \|r\|_2^2I_n  \\
              \end{bmatrix}\nonumber\\
& &\quad \times \begin{bmatrix}
                I_m & 0_{m\times n} \\
                -\frac{1}{\|r\|_2^2}x_Sr^T & I_n \label{eqblk1}\\
              \end{bmatrix}
\end{eqnarray}
and
\begin{eqnarray}\label{eqblk2}
I_m+\|x_S\|_2^2\left(I_m-\frac{1}{\|r\|_2^2}rr^T\right)=
\begin{bmatrix}
I_m & \|x_S\|_2\left(I_m-\frac{1}{\|r\|_2^2}rr^T\right) \\
\end{bmatrix}
\begin{bmatrix}
  I_m \\
  \|x_S\|_2\left(I_m-\frac{1}{\|r\|_2^2}rr^T\right) \\
\end{bmatrix},
\end{eqnarray}
we substitute \eqref{eqblk2} and \eqref{eqblk1} into \eqref{eqblk} and get
\begin{eqnarray}\label{eqblk4}
    KK^T=M^{-1}WW^TM^{-1},
\end{eqnarray}
where
\begin{eqnarray*}
    W=\begin{bmatrix}
        A^T & \|x_S\|_2A^T\left(I_m-\frac{1}{\|r\|_2^2}rr^T\right) & \|r\|_2\left(I_n-\frac{1}{\|r\|_2^2}A^Trx_S^T\right) \\
      \end{bmatrix}.
\end{eqnarray*}
Using the formula $\left\|K\right\|_2=\left\|KK^T\right\|_2^{1/2}$ again, we complete the proof with \eqref{eqblk5}  and \eqref{eqblk4}. $\Box$
\end{proof}
\begin{remark}\label{rmk2.1}
From Theorem \ref{Thm.main}, we can see that the orders of the matrices in \eqref{eqblk0}, \eqref{eqform1} and \eqref{eqform2} are $n\times m(n+1)$, $n\times n$ and $n\times (2m+n)$, respectively. When $m$ and $n$ are comparable and large, the two equivalent forms given by Theorem \ref{Thm.main} no longer contain a Kronecker product, and thus have some superiorities in storage and practical computation. But as pointed out in \cite{BabG11} and \cite{High02}, the matrix cross product $KK^T$ may lead to large rounding errors, so $\kappa_{F2}(A,\lambda b)$ is preferable for numerical stability.
\end{remark}

It should be noted that when $\lambda=1$, we get the TLS problem from \eqref{stls}. Based on Theorems \ref{Thm2} and \ref{Thm.main}, different expressions of the condition number for TLS problem follow
\begin{eqnarray*}
    \kappa_{TLSF}(A, b)=\left\|M^{-1}\left(\left(\frac{2}{\|r\|_2^2}A^Trr^T-A^T\right)
\left(\begin{bmatrix}
  x_{\intercal}^T & -1 \\
\end{bmatrix}\otimes I_m\right)-\begin{bmatrix}
                                  I_n\otimes r^T & 0_{n\times m} \\
                                \end{bmatrix}\right)\right\|_2,
\end{eqnarray*}
\begin{eqnarray}
     \kappa_{TLSF1}(A, b)=\left\|M^{-1}\left((1+\|x_{\intercal}\|^2_2)A^TA-A^Trx_{\intercal}^T-x_{\intercal}r^TA+\|r\|_2^2I_n\right)M^{-1}\right\|_2^{\frac{1}{2}},
\end{eqnarray}
and
\begin{eqnarray}
    \kappa_{TLSF2}(A, b)&=\left\|M^{-1}\begin{bmatrix}
        A^T & \|x_{\intercal}\|_2A^T\left(I_m-\frac{1}{\|r\|_2^2}rr^T\right) & \|r\|_2\left(I_n-\frac{1}{\|r\|_2^2}A^Trx_{\intercal}^T\right) \\
      \end{bmatrix}\right\|_2.\label{eqform21}
\end{eqnarray}
We note that $\kappa_{TLSF1}(A, b)$ was an intermediate result of Theorem 1 in \cite[Equation 3.8]{BabG11}, and $\kappa_{TLSF}(A, b)$ is given by Jia and Li \cite[Theorem 2]{JiaL13}. Based on $Mx_{\intercal}=A^Tb$ and its variants,  Baboulin and Gratton \cite{BabG11} also showed that
\begin{eqnarray}
    \kappa_{TLSF1}(A, b)&=\left\|(1+\|x_{\intercal}\|_2^2)M^{-1}\left(A^TA+\sigma_{n+1}\left(I_n-\frac{2}{1+\|x_{\intercal}\|_2^2}x_{\intercal}x_{\intercal}^T\right)\right)M^{-1}\right\|_2^{\frac{1}{2}},\label{eqform21}
\end{eqnarray}
and suggested that when the TLS problem is solved by the SVD method, the computation of \eqref{eqform21} can be further simplified. But their simplified expression needs the SVDs of both $A$ and $[A,b]$, this may be expensive. Jia and Li \cite{JiaL13} further showed that only the SVD of $[A,b]$ will be enough.  Based on the SVDs of $A$ and/or $[A,b]$, some computable upper and lower bounds of the condition number were also presented in \cite{BabG11} and \cite{JiaL13}.  Furthermore, it can be easily checked that
\begin{eqnarray*}
  A^TA+\sigma_{n+1}\left(I_n-\frac{2}{1+\|x_{\intercal}\|_2^2}x_{\intercal}x_{\intercal}^T\right)
\end{eqnarray*}
is positive definite. Xie et al. \cite[Remark 2]{Xie14} suggested that to use Cholesky decomposition to further simplify the expression of condition number, but no explicit expression was given there. According to Remark \ref{rmk2.1}, our new expression $\kappa_{TLSF2}(A, b)$ needs less storage space, and does not need to calculate Cholesky decomposition. We only need to calculate the product of matrices and vectors, since $M^{-1}$ can be the intermediate result when the TLS problem is solved with its normal equation. So we may say that the $\kappa_{TLSF2}(A, b)$ is a new result on the condition number of TLS problem, and enjoys storage and computational advantages.

As in \cite{Zhou09} and \cite{LiJ11}, when $\lambda\rightarrow 0$, we get $\sigma_{n+1}\rightarrow 0$. Therefore, $(A^TA-\sigma_{n+1})^{-1}\rightarrow (A^TA)^{-1}$ and $x_S$ converges to $x_\mathrm{O}$. When the matrix $A$ has full column rank, from Theorems \ref{Thm2} and \ref{Thm.main} and by the equality $A^Tr=0$, we get the following three equivalent expressions of the condition number for OLS problem
\begin{eqnarray*}
        \kappa_{OLSF}(A, b)=\left\|(A^TA)^{-1}\left(-A^T
\left(\begin{bmatrix}
  x_\mathrm{O}^T & -1 \\
\end{bmatrix}\otimes I_m\right)-\begin{bmatrix}
                                  I_n\otimes r^T & 0_{n\times m} \\
                                \end{bmatrix}\right)\right\|_2,
\end{eqnarray*}
\begin{eqnarray}\label{cdOLS1}
    \kappa_{OLSF1}(A,\lambda b)=\left\|(A^TA)^{-1}\left((1+\|x_\mathrm{O}\|^2_2)A^TA+\|r\|_2^2I_n\right)(A^TA)^{-1}\right\|_2^{\frac{1}{2}},
\end{eqnarray}
and
\begin{eqnarray}\label{cdOLS2}
   \kappa_{OLSF2}(A,\lambda b)=\left\|(A^TA)^{-1}\begin{bmatrix}
        A^T & \|x_\mathrm{O}\|_2A^T & \|r\|_2I_n\\
      \end{bmatrix}\right\|_2.
\end{eqnarray}
With a little algebra, we can check that $\kappa_{OLSF}(A, b)$ can be rewritten as follows
\begin{eqnarray}\label{cdOLS}
        \kappa_{OLSF}(A, b)=\left\|\begin{bmatrix}
                                     -(x_\mathrm{O}^T\otimes A^{\dagger})-(A^TA)^{-1}\otimes r^T & A^{\dagger} \\
                                   \end{bmatrix}\right\|_2,
\end{eqnarray}
where $A^{\dagger}=(A^TA)^{-1}A^T$ is the Moore-Penrose inverse of matrix $A$ (see \cite{BenI03,WangWQ04}). It should be noted that \eqref{cdOLS}, \eqref{cdOLS1} and \eqref{cdOLS2} have been given by Li and Wang \cite{LiW16} in investigating the condition number for indefinite least squares problem.
\begin{remark}
In \cite{LiJ11}, the authors also considered the linear structured condition number for the STLS problem. Although we can also make some progress on finding the compact form, like \eqref{eqform1}, of the linear structured condition number through the method given in \cite{LiW16b}, the final expression  may enjoy some computational advantage and is still very complicated. So in this paper we will not consider the normwise structured condition number for the STLS problem, for more research on structured condition number (see, e.g., \cite{Rump03a, Rump03b, DHigh92}).
\end{remark}

\section{Numerical experiment}
\label{Sect3}
In this part, we mainly focus on the computation of the condition number for STLS problem via its different forms. We note that the main task of calculating the condition number is to find the maximum eigenvalue of a matrix. For a large matrix, iterative techniques are always preferred in finding its extreme eigenvalues, a standard reference is \cite[Chapt.~10]{GoluC}.  Since the product of matrix and vector can be used to cancel the Kronecker product which coincides with the main step of the power method \cite[pp.~365]{GoluC},  Baboulin and  Gratton \cite{BabG11} proposed an iteration of the power method to compute the condition number for TLS problem. Similar to \cite[Proposition~3]{BabG11}, an iterative procedure can also be established.

To apply the power method, from equation \eqref{eqblk0} we get
\begin{eqnarray*}
    K^Ty=\left(\left(\frac{2}{\|r\|_2^2}A^Trr^T-A^T\right)
\left(\begin{bmatrix}
  x_S^T & -1 \\
\end{bmatrix}\otimes I_m\right)-\begin{bmatrix}
                                  I_n\otimes r^T & 0_{n\times m} \\
                                \end{bmatrix}\right)^TM^{-1}y.
\end{eqnarray*}
Since computing the inverse of a matrix is not desired, we may set $z=M^{-1}y$ and get $z$ by solving the linear system $Mz=y$. Thus, we can obtain
\begin{eqnarray*}
K^Ty&=&\left(\left(\begin{bmatrix}
                    x \\
                    -1 \\
                  \end{bmatrix}\otimes I_m\right)
                  \left(\frac{2}{\|r\|_2^2}rr^TA-A\right)-
                  \begin{bmatrix}
                    I_n\otimes r \\
                    0_{m\times n} \\
                  \end{bmatrix}
\right)z\\
&=&\mathrm{vec}\left(\begin{bmatrix}
         wx^T & -w \\
       \end{bmatrix}\right)-
       \mathrm{vec}\left(
       \begin{bmatrix}
         rz^T & 0_{m\times 1} \\
       \end{bmatrix}\right)\\
&=&\mathrm{vec}\left(\begin{bmatrix}
         wx^T-rz^T & -w \\
       \end{bmatrix}\right),
\end{eqnarray*}
where $w=\left(\frac{2}{\|r\|_2^2}rr^TA-A\right)z$.
We present the following algorithm for computing the condition number \eqref{eqblk0}, which circumvents the Kronecker product. The derivation of Algorithm \ref{Algpw} is very similar to the Algorithm 1 in \cite{BabG11}, so we omit some details.

\begin{algorithm}[htbp]
\caption{Power method}\label{Algpw}
Given the initial $y\in \mathbb{R}^n$.\\
\textbf{for }$\mathbf{i=1,2,\cdots}$
\begin{enumerate}
  \item $\begin{bmatrix}
          A_p & b_p \\
        \end{bmatrix}\leftarrow\begin{bmatrix}
         wx^T-rz^T & -w \\
       \end{bmatrix}$
  \item $v\leftarrow \left\|\begin{bmatrix}
          A_p & b_p \\
        \end{bmatrix}\right\|_F$
  \item $\begin{bmatrix}
          A_p & b_p \\
        \end{bmatrix}\leftarrow\frac{1}{v}
        \begin{bmatrix}
          A_p & b_p \\
        \end{bmatrix}$
  \item $y\leftarrow M^{-1}\left(\left(\frac{2}{\|r\|_2^2}A^Trr^T-A^T\right)
\left(A_px-b_p\right)-A_p^Tr\right)$
\end{enumerate}
\textbf{end}\\
$\kappa_{pwF}(A,\lambda b)=\sqrt{v}.$
\end{algorithm}

\begin{remark}
\label{rmk3.1}
Algorithm \ref{Algpw} is used to compute $v$, the maximum eigenvalue of $KK^T$, so we use $\sqrt{v}$ as the condition number of STLS problem. Moreover, the power method converges if $v$ is dominant and the initial vector $y$ has a component in the direction of the corresponding dominant eigenvector \cite[pp.~366]{GoluC}. The choice of $y$ is usually implied by applications or by using random vector.
\end{remark}

We also present a probabilistic condition estimation (PCE) method. This method is based on the probabilistic spectral norm estimator  proposed by Hochstenbach in \cite{R_Hochstenbach13}, which gives an interval containing the spectral norm of a matrix with high probability. The PCE method has been used to estimate the condition number of various problems (see \cite{LiW16,Wang15}).  For the convenience of presentation, we summarize the method given in \cite{R_Hochstenbach13} as the following lemma.
\begin{lemma}\label{LemPCE}
Let $\mathcal{U}(S_{p-1})$ be the uniform distribution over unit sphere $S_{p-1}$ in $R^{p}$, and $A\in \mathbb{R}^{m\times n}$. If we choose a random vector $z$ from $\mathcal{U}(S_{p-1})$, then by the probabilistic spectral norm estimator we have
\begin{eqnarray*}
  \alpha &\leq & \|A\|_2, \\
   \|A\|_2 &\leq & \beta \textrm{ with probability at least } 1-\epsilon, \\
  {\beta}/{\alpha} &\leq & 1+\theta,
\end{eqnarray*}
where $\alpha$ is the guaranteed lower bound, $\beta$ is the probabilistic upper bound, $\epsilon$ and $\theta$ are two use-chosen parameters.
\end{lemma}

Based on the above lemma, we can propose the following algorithm to give a sharp estimate of the condition number for the STLS problem.
\begin{algorithm}[htbp]
\caption{PCE method}\label{Algpce}
\begin{enumerate}
  \item Getting the start vector $v\in \mathbb{R}^{2m+n}$ from $\mathcal{U}(S_{2m+n-1})$.
  \item Compute the matrix in \eqref{eqform2} and let
\begin{align*}
    \widehat{K}=M^{-1}\begin{bmatrix}
        A^T & \|x_S\|_2A^T\left(I_m-\frac{1}{\|r\|_2^2}rr^T\right) & \|r\|_2\left(I_n-\frac{1}{\|r\|_2^2}A^Trx_S^T\right)
        \end{bmatrix}.
\end{align*}
  \item Compute $\alpha$ and $\beta$ of $\widehat{K}$ by probabilistic spectral norm estimator \cite{R_Hochstenbach13}.
  \item $\kappa_{pceF}(A,\lambda b)=\frac{\alpha+\beta}{2}.$
\end{enumerate}
\end{algorithm}

It should be pointed out that Step 2 can be done by the Matlab function \texttt{normprob.m} which can be downloaded from \url{http://www.win.tue.nl/~hochsten/eigenvaluetools/} . In the practical implementation of Algorithm \ref{Algpce} and as suggested in \cite{R_Hochstenbach13}, explicitly forming matrix $\widehat{K}$ may not be necessary, because what we really need is the product of a random vector with $\widehat{K}$ and $\widehat{K}^T$. Hence again, some techniques in solving linear system can be employed to reduce the computational burden, especially for large scale problems. More derivation and extension on the probabilistic condition estimation can be found in \cite{R_Hochstenbach13, Gaaf15}.

In the recent paper \cite{DiaoW16}, the authors proposed to use the small sample condition estimation (SCE) method \cite{KL94} to estimate the condition numbers of TLS problem, and proposed two ways to estimate the normwise condition number. The SCE method used in \cite{DiaoW16} can be directly applied to estimate the normwise condition number of STLS problem, so we adapt their Algorithm 2 needing less CUP time than Algorithm 1 with some modifications as Algorithm~\ref{AlgSCE}.
\begin{algorithm}[htbp]{
\caption{SCE method}\label{AlgSCE}
\begin{enumerate}
  \item {Generate $k$ vectors $z_1$, $z_2,\cdots,z_k\in \mathbb{R}^n$ with entries in the uniform continuous distribution on the interval $(0,1)$, where $k$ is the sample size. Othonomalize these vectors via modified Gram-Schmidt orthogonalization process.}
  \item Approximate the Willis factors $\omega_n$ and $\omega_k$ by
  \begin{equation*}
    \omega_n\approx \sqrt{\frac{2}{\pi(p-\frac{1}{2})}}\quad \textrm{and} \quad   \omega_k\approx \sqrt{\frac{2}{\pi(k-\frac{1}{2})}}.
  \end{equation*}
  \item For $i=1,\cdots,k$, compute
  \begin{equation*}
    \kappa_{i}=\left\|z_i^TM^{-1}\left((1+\|x_S\|^2_2)A^TA-A^Trx_S^T-x_Sr^TA+\|r\|_2^2I_n\right)M^{-1}z_i\right\|_2^{\frac{1}{2}}.
  \end{equation*}
  \item Estimate the absolute normwise condition number by
  \begin{eqnarray*}
{\kappa} _{sceF}(A,b) = \frac{\omega_k}{\omega_n}\sqrt{\sum_{i=1}^k\kappa_i^2} .
\end{eqnarray*}
\end{enumerate}}
\end{algorithm}

\begin{remark}
In the implementation of Algorithm \ref{AlgSCE}, we need to compute $M^{-1}z_i$, which is usually done by solving the linear system
\begin{eqnarray*}
My=z_i.
\end{eqnarray*}
Since $M$ is positive definite, the preconditioned conjugate gradient(PCG) method can be employed \cite{Bjo00}. Moreover, in practice the sample size $k=3$ will give a relative high accurate estimation of condition number.  Diao et al. \cite{DiaoW16} showed that the computational cost of SCE method is $O(mn+n^2)$. Here, we need to make some comments on the computational complexity of Algorithms \ref{Algpce} and \ref{AlgSCE}. As pointed out in \cite[Section~4.1]{LiW16b}, it is not easy to give an exact comparison of the computational complexity of these two algorithms. Because the probabilistic spectral norm estimator is based on Lanczos iteration method, and the dimension of Krylov space is automatically determined by $\epsilon$ \cite{R_Hochstenbach13}, whereas the SCE method mainly depends on matrix-vector product and one orthonomalization procedure. So instead of counting flops we report the CUP time to compare the efficiency of Algorithms \ref{Algpce} and \ref{AlgSCE}.

\end{remark}

\begin{example}
Since investigating the influence of different forms on the computation of condition number for STLS problem is our purpose, we construct the following random STLS problem, which is similar to \cite{BabG11}. Let $[A\; \lambda b]$ be defined by
\begin{eqnarray*}
    \begin{bmatrix}
       A & \lambda b \\
     \end{bmatrix}=Y\begin{bmatrix}
                      D \\
                      0 \\
                    \end{bmatrix}Z^T\in \mathbb{R}^{m\times (n+1)},
                    \; Y=I_m-2yy^T,\;Z=I_{n+1}-2zz^T,
\end{eqnarray*}
where $y\in \mathbb{R}^{m}$, $z\in \mathbb{R}^{n+1}$ are random unit vectors, and $D=\mathrm{diag}(n,n-1,\cdots,1, 1-e_p)$ for given parameter $e_p$. Due to the interlacing property \cite[pp.~178]{Bjor96}, we get
\begin{eqnarray*}
    \hat{\sigma}_n-\sigma_{n+1}\leq \sigma_n-\sigma_{n+1}=e_p.
\end{eqnarray*}
Thus $e_p$ gives a measure of the distance of the problem to nongenericity, and the solution $x_S$ is given by \eqref{xS}. By varying $\lambda$, $e_p$ and the order of matrix, we report the CUP time in seconds for computing the condition number of STLS problem with different forms. All the computations are performed in Matlab R2010b on a PC with Intel i5-6600M CPU 3.30 GHz and 4.00 GB RAM.

First, we compare two "naive" methods, that is, we first compute the explicit form of the matrices in \eqref{eqblk0} and \eqref{eqform2} and then compute its spectral norms by the built-in function \texttt{norm($\cdot$,2)}. We repeat it 200 times for one group of settings, and report the mean values of CPU time in Table \ref{Table1}.
\begin{table}[htp]
\centering
\small{
  \caption{Average CPU time in seconds for two "naive" methods}\label{Table1}
  \begin{tabular}{llccc}
    \hline
      &   & $m=100,n=70$ & $m=200,n=150$ & $m=500,n=300$ \\
    \hline
        &   & $\kappa_F(A,\lambda b)|\kappa_{F2}(A,\lambda b)$ & $\kappa_F(A,\lambda b)|\kappa_{F2}(A,\lambda b)$ & $\kappa_F(A,\lambda b)|\kappa_{F2}(A,\lambda b)$  \\
    \hline
    $\lambda=0.05$ & $e_p=0.1$  &0.0442$|$0.0025  &  1.0020$|$0.0086  & 13.3904$|$0.0711  \\
                   & $e_p=0.001$&0.0445$|$0.0022  &  1.0730$|$0.0114  & 12.8509$|$0.0612  \\
    \hline
      $\lambda=5$ & $e_p=0.1$   &0.0441$|$0.0022  &  1.0113$|$0.0085  & 13.0997$|$0.0671 \\
                  & $e_p=0.001$ &0.0436$|$0.0029  &  1.0597$|$0.0107  & 12.9750$|$0.0670\\
    \hline
  \end{tabular}}
\end{table}
From Table~\ref{Table1}, we can see that when $m=500,n=300$, computing \eqref{eqblk0} becomes very time consuming due to the large order of matrix, but \eqref{eqform2} still works well. Moreover, we can also find that the CPU time for computing \eqref{eqblk0} is always smaller than that for \eqref{eqform2}.

Second, Algorithm \ref{Algpw} provides a method for avoiding the Kronecker product. So with respect to accuracy and running time, we give a comparison of the efficiency of computing condition number for four different methods: (1) use the formula \eqref{eqform2} exactly(EXA); (2)  use Algorithm \ref{Algpw}(PW). (3) use Algorithm \ref{Algpce}(PCE). For Algorithm \ref{Algpce}, the user-chosen parameters are given by $\epsilon=0.001$, $\delta= 0.01$; (4) use Algorithm \ref{AlgSCE}, and the sample size $k=3$. In Algorithm \ref{Algpw},  the initial vector $y$ is a random vector with elements from standard normal distribution, and the algorithm terminates when the difference between two successive values of $v$ is less than $10^{-8}$ or the number of iterations exceeds 500. To show the accuracy of estimation, we use formula \eqref{eqform2} as the bench mark, and define the following ratios
\begin{eqnarray*}
\rm ratio1=\frac{PW}{EXA},\quad ratio2=\frac{PCE}{EXA}, \quad ratio3=\frac{PCE}{EXA}.
\end{eqnarray*}
The ratios are plotted in Figure \ref{Fig.ratio}, we only report the case $\lambda=5$ and $e_p=0.1$, since the results for other cases are similar.
\begin{figure}[htbp]
  \centering
  \includegraphics[width=0.67\textwidth,height=0.5\textwidth]{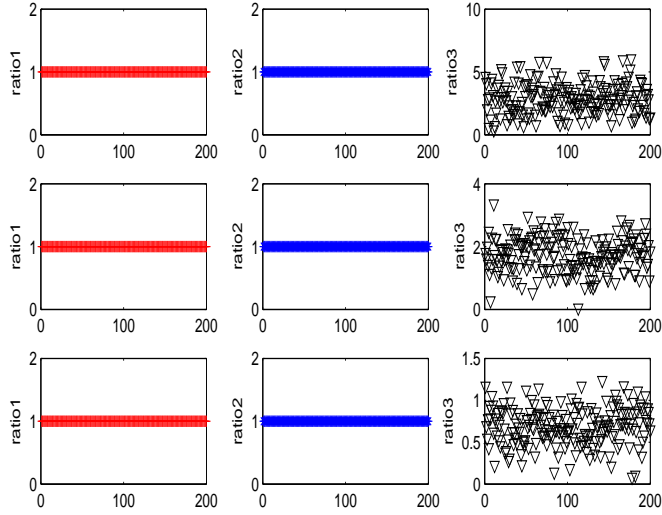}\\
  \caption{Accuracy of condition number estimators; the first row $m=200,n=150$, the second row $m=500,n=300$, the third row $m=1000,n=700$.}\label{Fig.ratio}
\end{figure}
From Figure \ref{Fig.ratio}, we can see that both PW and PCE methods give very accurate estimates of the condition number. The SCE method also gives acceptable estimates, since the ratios are contained in the interval $(0.1,10)$ \cite{High02}. This coincides with the results in \cite{DiaoW16}. Thus, if the accuracy of estimation is required, PW and PCE methods are preferred.

Now, we turn to the running time of these four different methods with 200 replications for each case. The numerical results are presented in Table~\ref{Table2}.
\begin{table}[htp]
\centering
\small{
  \caption{Average CPU time in seconds for EXA, PW, PCE and SCE methods}\label{Table2}
  \begin{tabular}{lllccc}
    \hline
                   &            &       & $m=200,n=150$ & $m=500,n=300$ & $m=1000,n=700$ \\
    \hline
    $\lambda=0.05$ & $e_p=0.1$  & EXA   & 0.0073 & 0.0491  & 0.4071 \\
                   & $e_p=0.1$  & PW    & 0.0074 & 0.0843  & 0.5686 \\
                   & $e_p=0.1$  & PCE   & 0.0033 & 0.0051  & 0.0218 \\
                   & $e_p=0.1$  & SCE   & 0.0020 & 0.0111  & 0.1108 \\
    \hline
                   & $e_p=0.001$&  EXA  & 0.0070 &  0.0495 & 0.4093 \\
                   & $e_p=0.001$&  PW   & 0.1502 &  1.5810 & 11.9597\\
                   & $e_p=0.001$&  PCE  & 0.0122 &  0.0136 & 0.0299 \\
                   & $e_p=0.001$&  SCE  & 0.0019 &  0.0114 & 0.1096 \\
    \hline
    $\lambda=5$    & $e_p=0.1$  & EXA   & 0.0071  & 0.0487 & 0.3810 \\
                   & $e_p=0.1$  & PW    & 0.0073  & 0.0811 & 0.5421 \\
                   & $e_p=0.1$  & PCE   & 0.0033  & 0.0049 & 0.0211 \\
                   & $e_p=0.1$  & SCE   & 0.0019  & 0.0110 & 0.1136 \\
    \hline
                   & $e_p=0.001$&  EXA  & 0.0072  & 0.0487 & 0.4409 \\
                   & $e_p=0.001$&  PW   & 0.0056  & 0.0649 & 0.4861 \\
                   & $e_p=0.001$&  PCE  & 0.0118  & 0.0141 & 0.0298 \\
                   & $e_p=0.001$&  SCE  & 0.0021  & 0.0111 & 0.1043 \\
    \hline
  \end{tabular}}
\end{table}
From Table~\ref{Table2}, we can find that, in most cases, these four methods perform very efficient in estimating the condition number of STLS problem. We also note that the PCE method is the most efficient especially for the large order cases, and EXA, PW and SCE have comparable performance, except for the case $\lambda=0.05$ and $e_p=0.001$, in which the variances of PW's CPU time are $0.3457$, $16.5316$ and $758.6435$. An explicit derivation for the underlying reason may be unavailable now, so we give some investigation through simulation. According to Remark \ref{rmk3.1}, we generate 20 groups of data, and use 100 different random vectors as the initial vector for each group of data to test its influence on the CPU time of PW method. We give the box-plot of the results as Figure \ref{Fig1}.
\begin{figure}[htbp]
  \centering
  \includegraphics[width=0.67\textwidth,height=0.5\textwidth]{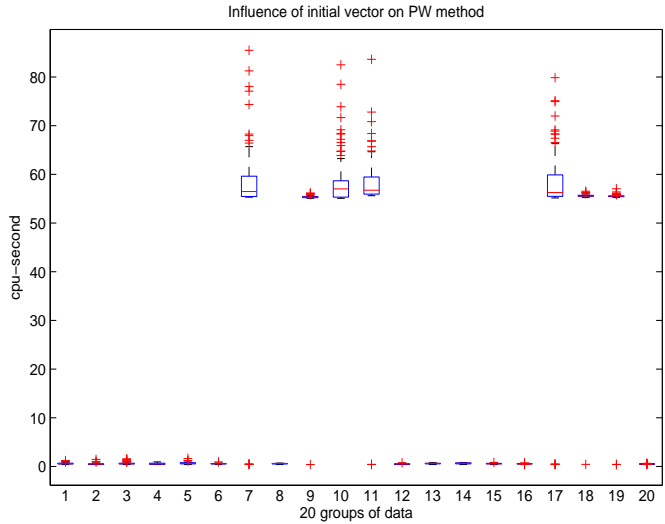}\\
  \caption{Box-plot for $\lambda=0.05$, $e_p=0.001$, $m=1000$ and $n=700$.}\label{Fig1}
\end{figure}
From Figure \ref{Fig1}, we find that when $m$ and $n$ are large, the PW method may perform unstable for some groups of data and has a lot of outliers which lead to its large mean values. However, we cannot conclude that the PW method is inefficient in estimating the condition number of STLS problem. Because, for most groups of data and with variant initial vectors, the PW method performs quite well. This may imply that our construction tends to give an ill-posed STLS problem for small $\lambda$ and $e_p$.

From our numerical experiment, we suggest that for moderate scale STLS problems computing the condition number via \eqref{eqform2} is recommended. The reason is that compared with \eqref{eqblk0}, \eqref{eqform2} not only avoids computing a Kronecker product and saves storage space but also needs less CPU time and preserves high accuracy.  When the coefficient matrix of the STLS problem is large, the PCE method can give highly accurate estimates of the condition number and needs less CPU time.
\end{example}

\section{Concluding remark}
\label{Sect4}
In this paper, we present some new expressions of the condition number for the STLS problem. The new expressions do not contain a Kronecker product, and make it possible to store the condition number in the computer for large scale problems. The new and compact forms of the condition numbers for the STLS and TLS problem are of certain interest for the practitioners from other areas like engineering, statistics and so on. This is because the new forms need less storage space and is very easy to use. In addition, to avoid explicitly forming the matrix in the expression of condition number, some iterative methods are also introduced. We also present some numerical experiments to check the proposed algorithms, and find that our algorithms have very good performance for most of our settings. However, in our experiment, we note that when $\lambda$ and $e_p$ are small, and the matrix is large, the power method can be very inefficient. Due to the difficulty of theoretical justification, only some simulations are given to explore the underlying reason, which is not enough and further investigation should be carried out in the future. To be on the safe side, we recommend using $\kappa_{F2}(A,\lambda b)$ or PCE method to compute the condition number of STLS problem in practical applications.


\end{document}